\documentclass{amsart}

\usepackage{pstricks}

\usepackage{hyperref}

\newtheorem{theorem}{Theorem}[section]
\newtheorem{proposition}[theorem]{Proposition}
\newtheorem{lemma}[theorem]{Lemma}

\newcommand{\M}{\mathbb{M}}
\newcommand{\R}{\mathbb{R}}

\newcommand{\ang}{\angle}
\newcommand{\angkap}{\angle^\kappa}
\renewcommand{\b}{\bar}
\newcommand{\cs}{\operatorname{cs}}
\newcommand{\diam}{\operatorname{diam}}
\newcommand{\es}{\emptyset}
\newcommand{\gam}{\gamma}
\newcommand{\kap}{\kappa}
\newcommand{\md}{\operatorname{md}}
\newcommand{\om}{\omega}
\newcommand{\ol}{\overline}
\newcommand{\olH}{{\,\ol{\!H}}}
\newcommand{\per}{\operatorname{per}}
\newcommand{\sm}{\setminus}
\newcommand{\sn}{\operatorname{sn}}
\newcommand{\sub}{\subset}
\newcommand{\til}{\tilde}

\begin{document}

\title[Toponogov's theorem for Alexandrov spaces]%
{On Toponogov's comparison theorem for Alexandrov spaces}

\author{Urs Lang}
\address{Departement Mathematik, ETH Z\"urich,
R\"amistrasse 101, 8092 Z\"urich, Switzerland}
\email{lang@math.ethz.ch}

\author{Viktor Schroeder}
\address{Institut f\"ur Mathematik, Universit\"at Z\"urich,
Winterthurerstrasse 190, 8057 Z\"urich, Switzerland}
\email{viktor.schroeder@math.uzh.ch}

\date{25 July 2012}

\maketitle


\section*{Introduction}

In this expository note, we present a transparent proof of Toponogov's theorem 
for Alexandrov spaces in the general case, not assuming local compactness 
of the underlying metric space. More precisely, we show that if $M$ 
is a complete geodesic metric space such that the Alexandrov triangle 
comparisons for curvature greater than or equal to $\kap \in \R$ 
are satisfied locally, then these comparisons also hold in the large; 
see Theorem~\ref{Thm:top}. The core of the proof is Proposition~\ref{Prop:45}. 
It states that a hinge $H = px \cup py$ in $M$ has the desired comparison 
property if every hinge $H' = p'x' \cup py'$ with an endpoint on $H$ and 
perimeter $|p'x'| + |p'y'| + |x'y'|$ less than some fixed fraction of the 
perimeter of $H$ has this property. The argument involves simple 
inductive constructions in $M$ and the model space $\M^2_\kap$ of constant 
curvature, leading to two monotonic quantities 
(see~\eqref{eq:l} and~\eqref{eq:xy}), whose limits agree.
This immediately gives the required inequality. 

The history of Toponogov's theorem starts with the work of 
Alexandrov~\cite{Alex1}, who proved it for convex surfaces.
Toponogov~\cite{Top1,Top2,Top3} established the result for 
Riemannian manifolds, in which case the local comparison inequalities
are equi\-valently expressed as a respective lower bound on 
the sectional curvature. 
A first purely metric local-to-global argument was given in~\cite{Pla1} for 
geodesic metric spaces with extendable geodesics.
In its most general form, without the assumption of local compactness,
the theorem was proved in~\cite{BurGP} (in~\cite{BurBI} the result is 
attributed to Perelman). An independent approach, building 
on~\cite{Pla1}, was then described by Plaut~\cite{Pla2}. The present note
further simplifies his argument.

In fact, the statements in both~\cite{BurGP} and~\cite{Pla2} differ from 
what is shown here in that the metric of $M$ is merely assumed
to be intrinsic (that is, $d(p,q)$ equals the infimum of the lengths
of all curves connecting $p$ and $q$, but it is not required that the 
infimum is attained); correspondingly, the Alexandrov comparisons are 
formulated without reference to shortest curves in~$M$.
However, assuming $M$ to be geodesic is not a severe restriction.
By~\cite[Theorem~1.4]{Pla2}, for every point $p$ in a complete, 
intrinsic metric space $M$ of curvature locally bounded below there is a 
dense G$_\delta$ subset $J_p$ of $M$ such that for all $q \in J_p$ there 
exists a shortest curve from $p$ to $q$. A proof of Toponogov's 
theorem for intrinsic spaces via essentially the same construction as here, 
which was obtained independently by Petrunin, is contained in the preliminary 
version of the forthcoming book~\cite{AleKP}.  
Nevertheless, we felt that it would be worthwhile to make the argument
in the present sleek form for geodesic spaces (such as complete Riemannian 
manifolds) available in the literature.

\section{Preliminaries}

In this section we fix the notation and recall some basic definitions 
and facts from metric geometry.

Let $M$ be a metric space with metric $d$. 
By a {\em segment\/} connecting two points $p,q$ in $M$ we mean the image 
of an isometric embedding $[0,d(p,q)] \to M$ that maps $0$ to $p$ and
$d(p,q)$ to $q$. We will write $pq$ for some such segment (assuming there 
is one), despite the fact that it need not be uniquely determined by $p$ 
and $q$.
We will use the symbol $|pq|$ as a shorthand for $d(p,q)$, regardless 
of the existence of a segment $pq$.
By a {\em hinge\/} $H = H_p(x,y)$ 
in $M$ we mean a collection of three points $p,x,y$ and two nondegenerate
segments $px,py$ in $M$; thus $p \not\in \{x,y\}$ (but possibly $x = y$). 
We call $p$ the {\em vertex}, $x,y$ the {\em endpoints}, and $px,py$ the 
{\em sides\/} of $H$.
The {\em perimeter\/} of a triple $(p,x,y)$ of points in $M$ is the number
\[
\per(p,x,y) := |px| + |py| + |xy|.
\] 
By the perimeter $\per(H)$ of a hinge $H = H_p(x,y)$ we mean 
the perimeter of the triple $(p,x,y)$.

We denote by $\M^m_\kap$ the $m$-dimensional, complete and simply connected
model space of constant sectional curvature $\kap \in \R$.
We write 
\[
D_\kap := \diam(\M^m_\kap) 
= \begin{cases} 
\pi/\sqrt{\kap} & \text{if $\kap > 0$,} \\
\infty & \text{if $\kap \le 0$}
\end{cases}
\]
for the diameter of $\M^m_\kap$.
Some trigonometric formulae for the model spaces are
collected in the appendix. The following basic monotonicity property
follows readily from the law of cosines, 
equation~\eqref{eq:snc}.

\begin{lemma} \label{Lem:c-ab}
Let $\kap \in \R$, and let $a,b \in (0,D_\kap)$ be fixed. 
For $\gam \in [0,\pi]$, let $H_p(x,y)$ be a hinge in $\M^2_\kap$ with 
$|px| = b$ and $|py| = a$ such that the hinge angle $\ang_p(x,y)$ 
(between $px$ and $py$) equals $\gam$, and put $c_{a,b}(\gam) := |xy|$. 
The function $c_{a,b}$ so defined is continuous and strictly increasing 
on~$[0,\pi]$.
\end{lemma}

The next lemma goes back to Alexandrov~\cite{Alex1}, 
compare~\cite[Lemma~4.3.3]{BurBI}.

\begin{lemma} \label{Lem:alexandrov}
Suppose that $H_p(q,y)$ and $H_q(x,y)$ are two hinges in $\M^2_\kap$ 
with $|py|,|qy|,|pq|+|qx| < D_\kap$, and $H_{\ol p}(\ol x,\ol y)$ is a hinge 
in $\M^2_\kap$ such that $|\ol p\,\ol x| = |pq|+|qx|$, $|\ol p\,\ol y| = |py|$,
and $|\ol x\,\ol y| = |xy|$.
Then $\ang_q(p,y) + \ang_q(x,y) \le \pi$ if and only 
if $\ang_p(q,y) \ge \ang_{\ol p}(\ol x,\ol y)$, and 
$\ang_q(p,y) + \ang_q(x,y) \ge \pi$ if and only if 
$\ang_p(q,y) \le \ang_{\ol p}(\ol x,\ol y)$.
\end{lemma}

\begin{proof}
Prolongate $pq$ to a segment $px'$ of length $|px'| = |pq|+|qx|$;
see Figure~\ref{Fig:alexandrov}. Consider the following obvious 
identities: 
\begin{align}
\pi - \ang_q(p,y) - \ang_q(x,y) 
&= \ang_q(x',y) - \ang_q(x,y), \label{eq:al-1} \\ 
|x'y| - |xy| 
&= |x'y| - |\ol x\,\ol y|, \label{eq:al-2} \\
\ang_p(x',y) - \ang_{\ol p}(\ol x,\ol y) 
&= \ang_p(q,y) - \ang_{\ol p}(\ol x,\ol y). \label{eq:al-3}
\end{align}
By Lemma~\ref{Lem:c-ab}, the right side of~\eqref{eq:al-1} and 
the left side of~\eqref{eq:al-2} have the same sign, 
and also the right side of~\eqref{eq:al-2} and 
the left side of~\eqref{eq:al-3} have equal sign. Hence, the same holds for 
the left side of~\eqref{eq:al-1} and the right side of~\eqref{eq:al-3}.
\end{proof}
\begin{figure}
\begin{center}
\psset{unit=1mm,linewidth=0.8pt}
\begin{pspicture}(0,0)(90,33)
\psdots(5,4)(30,4)(34,16)(26,25.8)(38,28)
\psline(5,4)(30,4)(34,16)
\psline(5,4)(34,16)(26,26)
\psline[linestyle=dotted,linewidth=1.2pt](5,4)(26,26)
\psline[linestyle=dashed](34,16)(38,28)
\uput[-135](5,4){$y$}
\uput[-45](30,4){$p$}
\uput[0](34,16){$q$}
\uput[45](26,26){$x$}
\uput[45](38,28){$x'$}
\psarc[linewidth=0.3pt](30,4){4}{71}{180}
\psarc[linewidth=0.3pt](34,16){4}{129}{202}
\psarc[linewidth=0.3pt](34,16){5}{202}{251}
\psdots(60,4)(85,4)(78.02,28.32)
\psline(60,4)(85,4)(78.02,28.32)
\psline[linestyle=dotted,linewidth=1.2pt](60,4)(78.02,28.32)
\uput[-135](60,4){$\ol y$}
\uput[-45](85,4){$\ol p$}
\uput[45](78.02,28.32){$\ol x$}
\psarc[linewidth=0.3pt](85,4){4}{106}{180}
\end{pspicture}
\end{center}
\caption{Proof of Lemma~\ref{Lem:alexandrov}} \label{Fig:alexandrov}
\end{figure}

Let again $M$ be a metric space, and let $\kap \in \R$.
Given $p,x,y \in M$, a triple $(\ol p,\ol x,\ol y)$ 
of points in~$\M^2_\kap$ is called a {\em comparison triple\/} 
for $(p,x,y)$ if $|\ol p\,\ol x| = |px|$, $|\ol p\,\ol y| = |py|$, 
and $|\ol x\,\ol y| = |xy|$. 
If $\kap \le 0$, such a comparison triple always 
exists, and if $\kap > 0$, a comparison triple exists if and only if  
$\per(p,x,y) \le 2D_\kap$. 
This is obvious if one of the distances
$a := |py|$, $b := |px|$, and $c := |xy|$ is zero or equal to
$D_\kap$. Otherwise, when $a,b,c \in (0,D_\kap)$, the assertion follows 
from Lemma~\ref{Lem:c-ab}: Depending on whether $a + b < D_\kap$ or 
$a + b \ge D_\kap$, the function $c_{a,b}$ maps $[0,\pi]$ bijectively 
onto $[|a-b|,a+b]$ or $[|a-b|,2D_\kap-a-b]$.
In either case, the given number $c$ is contained in the image of $c_{a,b}$, 
so there exists a unique $\gam \in [0,\pi]$ such that $c_{a,b}(\gam) = c$.

Now consider a triple $(p,x,y)$ of points in $M$ such that $p \not\in \{x,y\}$.
In case $\kap > 0$, suppose that $|px|,|py| < D_\kap$ and  
$\per(p,x,y) \le 2D_\kap$. Then any comparison 
triple $(\ol p,\ol x,\ol y)$ in $\M^2_\kap$ uniquely determines a hinge 
$H_{\ol p}(\ol x,\ol y)$ and one defines  
the {\em comparison angle\/} $\angkap_p(x,y) \in [0,\pi]$ as the
hinge angle, thus
\[
\angkap_p(x,y) := \ang_{\ol p}(\ol x,\ol y).
\]
For an arbitrary hinge $H_p(x,y)$ in $M$, the
(Alexandrov) {\em angle\/} or {\em upper angle\/} of $H_p(x,y)$ is 
then defined by
\[
\angle_p(x,y) := \limsup_{\substack{u \in px,\,v \in py\\u,v \to p}}
\angkap_p(u,v).
\]
The number $\angle_p(x,y)$ is clearly independent of $\kap \in \R$.
Furthermore, if $px,py,pz$ are three nondegenerate segments, 
the triangle inequality
\begin{equation} \label{eq:tri}
\angle_p(x,y) + \angle_p(y,z) \ge \angle_p(x,z)
\end{equation} 
holds, see~\cite{Alex0} or~\cite[Part I, Proposition~1.14]{BriH}.

Let again $H = H_p(x,y)$ be a hinge in $M$, and suppose that 
$\per(H) < 2D_\kap$. 
Let $(\ol p,\ol x,\ol y)$ be a comparison triple in $\M^2_\kap$ for $(p,x,y)$,
and let $H_{\hat p}(\hat x,\hat y)$ be a {\em comparison hinge\/} in~$\M^2_\kap$ 
for $H$, that is, $|\hat p\hat x| = |px|$, $|\hat p\hat y| = |py|$, and 
$\ang_{\hat p}(\hat x,\hat y) = \ang_p(x,y)$.
We are interested in the following comparison properties that $H$ may or 
may not have:
\begin{itemize}
\item[(A$_\kap$)] 
(Angle comparison) $\ang_p(x,y) \ge \angkap_p(x,y)$ 
($= \ang_{\ol p}(\ol x,\ol y)$); 
\item[(H$_\kap$)]
(Hinge comparison) $|xy| \le |\hat x\hat y|$;
\item[(D$_\kap$)]
(Distance comparison) $|uv| \ge |\ol u\,\ol v|$ whenever $u \in px$, 
$v \in py$, $\ol u \in \ol p\,\ol x$, $\ol v \in \ol p\,\ol y$, and 
$|pu| = |\ol p\,\ol u|$, $|pv| = |\ol p\,\ol v|$. 
\end{itemize}
It follows easily from Lemma~\ref{Lem:c-ab} that, for an individual hinge 
$H$ as above,
\[
\text{(D$_\kap$) $\Rightarrow$ (A$_\kap$) $\Leftrightarrow$ (H$_\kap$).}
\]
For the implication (A$_\kap$) $\Rightarrow$ (D$_\kap$), 
see Lemma~\ref{Lem:a-d} below. The metric space $M$ is called 
a {\em space of curvature~$\ge \kap$ in the sense of Alexandrov\/}
if every point $q$ has a neighborhood $U_q$ such that any two points 
in $U_q$ are connected by a segment in $M$ and every hinge $H = H_p(x,y)$ 
with $p,x,y \in U_q$ (and $\per(H) < 2D_\kap$) satisfies (D$_\kap$). 
Again due to Lemma~\ref{Lem:c-ab}, the upper angle
between two segments in such a space $M$ always exists as a limit, 
by monotonicity. We call a segment $px$ in a metric space {\em balanced\/} if,
for every nondegenerate segment $qy$ with $q \in px \sm \{p,x\}$,
the angles formed by $qy$ and the subsegments $qp,qx$ of $px$ 
satisfy $\ang_q(p,y) + \ang_q(x,y) = \pi$. Note that, by~\eqref{eq:tri},
the inequality $\ang_q(p,y) + \ang_q(x,y) \ge \pi$ always holds, 
since $\ang_q(p,x) = \pi$. Of course, in a Riemannian manifold every segment 
is balanced.

\begin{lemma} \label{Lem:a-d}
Let $\kap \in \R$, and let $M$ be a metric space. Then:
\begin{itemize}
\item[\rm (i)] 
If $M$ is a space of curvature $\ge \kap$ in the sense of 
Alexandrov, then all segments in $M$ are balanced.
\item[\rm (ii)]
Let $H = H_p(x,y)$ be a hinge in $M$ with balanced sides and 
$\per(H) < 2D_\kap$. Suppose that every pair of points in $px \cup py$ is 
connected by a segment in~$M$ and every hinge with one side contained 
in $px$ or $py$ and the opposite endpoint on the other side of~$H$
satisfies~{\rm (A$_\kap$)}. Then $H$ satisfies~{\rm (D$_\kap$)}.
\end{itemize}
\end{lemma}

\begin{proof}
For~(i), let $px,qy$ be two nondegenerate segments in $M$ such that 
$q \in px \sm \{p,x\}$. Let $u \in qp$, $v \in qx$, $w \in qy$ be points 
distinct from $q$, and assume that $u \ne w$. If $u,v,w$ are sufficiently close
to $q$, then there is a segment $uw$ such that the hinge $H_u(v,w)$ with 
$uv \sub px$ satisfies (D$_\kap$). Let $(\ol u,\ol v,\ol w)$ be a comparison 
triple in $\M^2_\kap$ for $(u,v,w)$, and let $\ol q \in \ol u\,\ol v$ be the 
point with $|\ol q\,\ol u| = |qu|$. Then $|qw| \ge |\ol q\,\ol w|$ and so 
$\angkap_u(q,w) \ge \ang_{\ol u}(\ol q,\ol w) = \ang_{\ol u}(\ol v,\ol w)$ by 
Lemma~\ref{Lem:c-ab}. Now Lemma~\ref{Lem:alexandrov} shows that
$\angkap_q(u,w) + \angkap_q(v,w) \le \pi$. 
Passing to the limit for $u,v,w \to q$
we get $\ang_q(p,y) + \ang_q(x,y) \le \pi$. 

We prove (ii). Let $(\ol p,\ol x,\ol y)$ be a comparison triple in $\M^2_\kap$ 
for $(p,x,y)$, and let $u,v$ and $\ol u,\ol v$ be given as in (D$_\kap$).
We first show that $|uy| \ge |\ol u\,\ol y|$. Omitting some trivial cases, we
assume $u \not\in \{p,x,y\}$. Choose a segment $uy$. Then 
$\angkap_u(p,y) + \angkap_u(x,y) \le \ang_u(p,y) + \ang_u(x,y) = \pi$
by the assumptions and so Lemma~\ref{Lem:alexandrov} yields 
$\angkap_p(u,y) \ge \ang_{\ol p}(\ol x,\ol y) = \ang_{\ol p}(\ol u,\ol y)$.
By Lemma~\ref{Lem:c-ab}, $|uy| \ge |\ol u\,\ol y|$. An analogous 
argument shows that $|uv| \ge |\til u\,\til v|$ if $(\til p,\til u,\til y)$ 
is a comparison triple for $(p,u,y)$ and $\til v \in \til p\,\til y$ is such
that $|pv| = |\til p\,\til v|$. Since $|\til u\,\til y| = |uy| 
\ge |\ol u\,\ol y|$, we have 
$\ang_{\til p}(\til u,\til v) = \ang_{\til p}(\til u,\til y) \ge 
\ang_{\ol p}(\ol u,\ol y) = \ang_{\ol p}(\ol u,\ol v)$ 
(assuming $p \not\in \{u,v\}$) 
and hence $|\til u\,\til v| \ge |\ol u\,\ol v|$ by Lemma~\ref{Lem:c-ab}. 
So $|uv| \ge |\ol u\,\ol v|$. 
\end{proof}


\section{The globalization theorem}

Now we prove Toponogov's theorem, in the form stated in 
Theorem~\ref{Thm:top} below. The central piece of the argument is
Proposition~\ref{Prop:45}, the following lemma and the concluding 
part of the proof are standard techniques.

\begin{lemma} \label{Lem:subdiv}
Let $\kap \in \R$, let $M$ be a metric space, and let $H = H_p(x,y)$ be 
a hinge in $M$ with $\per(H) < 2D_\kap$.
Suppose that there exist a point $q$ on $px$, distinct from $p,x,y$, 
and a segment $qy$ such that each of the three hinges 
$H_p(q,y),H_q(p,y),H_q(x,y)$ with sides in $px \cup py \cup qy$ 
satisfies {\rm (A$_\kap$)}, and $\ang_q(p,y) + \ang_q(x,y) = \pi$.
Then $H$ satisfies {\rm (A$_\kap$)} as well.
\end{lemma}

\begin{proof}
Note that $\per(p,q,y),\per(q,x,y) \le \per(H) < 2D_\kap$.
Since $H_p(q,y)$ satisfies (A$_\kap$), we have
$\ang_p(x,y) = \ang_p(q,y) \ge \angkap_p(q,y)$.
By the remaining assumptions,
$\angkap_q(p,y) + \angkap_q(x,y) \le \ang_q(p,y) + \ang_q(x,y) = \pi$
and so Lemma~\ref{Lem:alexandrov} gives 
$\angkap_p(q,y) \ge \angkap_p(x,y)$.
Thus $\ang_p(x,y) \ge \angkap_p(x,y)$.
\end{proof}

\begin{proposition} \label{Prop:45}
Let $\kap \in \R$, and let $M$ be a metric space such that every pair of 
points in $M$ at distance $< D_\kap$ is connected by a balanced segment.
Let $H_p(x,y)$ be a hinge in $M$ with balanced sides and
$\per(p,x,y) < 2D_\kap$. If every hinge $H_{p'}(x',y')$ in $M$ with 
balanced sides, $\per(p',x',y') < \frac45 \per(p,x,y)$, and 
$\{x',y'\} \cap (px \cup py) \ne \es$ satisfies {\rm (A$_\kap$)}, 
then $H_p(x,y)$ satisfies {\rm (A$_\kap$)} as well.
\end{proposition}

\begin{proof}
We prove the following assertion, from which the general result follows 
easily by a repeated application of Lemma~\ref{Lem:subdiv}:
{\it Let $H_0 = H_{p_0}(x_0,y_0)$ be a hinge in $M$ with balanced sides
and $|p_0x_0| < \min\bigl\{\tfrac15 |p_0y_0|, D_\kap - |p_0y_0|\bigr\}$.
If every hinge $H_{p'}(x',y')$ in $M$ with balanced sides,
$\per(p',x',y') < \frac45 \per(H_0)$, and 
$\{x',y'\} \cap \{x_0,y_0\} \ne \es$ satisfies {\rm (A$_\kap$)}, 
then $H_0$ satisfies {\rm (A$_\kap$)} as well.}
We put $a := |p_0y_0|$ and $b := |p_0x_0|$, so $b < \tfrac15 a$ and
$a + b < D_\kap$.

First, starting from $H_0$, we will inductively construct a
particular sequence of hinges $H_n = H_{p_n}(x_n,y_n)$ in $M$ with balanced 
sides such that $\{x_n,y_n\} = \{x_0,y_0\}$
and the numbers $l_n := |p_nx_n| + |p_ny_n|$ satisfy
\begin{equation} \label{eq:l}
a + b = l_0 \ge l_1 \ge l_2 \ge \ldots \ge |x_0y_0|;
\end{equation}
furthermore, for $n \ge 1$, $|p_nx_n| = b' := \frac25 a$ and hence
\[
|p_ny_n| \ge |x_ny_n| - |p_nx_n| = |x_0y_0| - b' \ge a - b - b' > b'.
\]
The hinge $H_0$ is already given.
For $n \ge 1$, if $H_{n-1}$ is constructed, let $p_n \in p_{n-1}y_{n-1}$ 
be the point at distance $b'$ from $y_{n-1}$, and put 
$x_n := y_{n-1}$ and $y_n := x_{n-1}$. Note that 
\[
|p_ny_n| \le |p_{n-1}p_n| + |p_{n-1}y_n| = l_{n-1} - b' 
\le a + b - b' < \tfrac45 a
\]
and hence $\per(p_{n-1},p_n,y_n) < \tfrac85 a \le \tfrac45 \per(H_0)$.
The sides of $H_n$ are the subsegment $p_nx_n$ of $p_{n-1}y_{n-1}$ and
an arbitrarily chosen balanced segment $p_ny_n$. Denote the angle of 
$H_n$ by $\gam_n$, and note that since $p_{n-1}y_{n-1}$ is balanced,
the adjacent angle between $p_ny_n$ and the subsegment $p_np_{n-1}$ 
of $p_{n-1}y_{n-1}$ equals $\pi - \gam_n$. See Figure~\ref{Fig:hn}. 
Clearly $l_n \le l_{n-1}$.
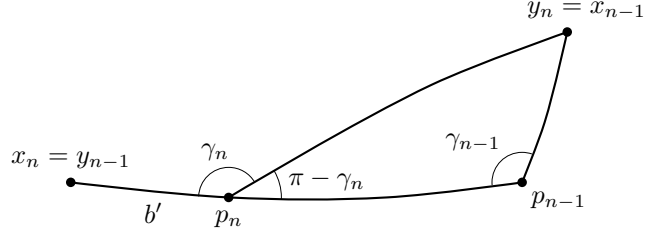
\begin{figure}
\begin{center}
\psset{unit=1mm,linewidth=0.8pt}
\begin{pspicture}(0,0)(90,31)
\psdots(12,6)(33,4)(72,6)(78,26)
\pscurve(12,6)(33,4)(54,4)(72,6)
\pscurve(72,6)(75,14)(78,26)
\pscurve(33,4)(60,19)(78,26)
\uput[90](12,6){$x_n = y_{n-1}$}
\uput[-90](23,5){$b'$}
\uput[-90](33,4){$p_n$}
\uput[-45](72,6){$p_{n-1}$}
\uput[75](78,26){$y_n = x_{n-1}$}
\psarc[linewidth=0.3pt](33,4){4}{30}{178}
\uput{5}[105](33,4){$\gam_n$}
\psarc[linewidth=0.3pt](33,4){7}{-2}{30}
\uput{8}[14](33,4){$\pi-\gam_n$}
\psarc[linewidth=0.3pt](72,6){4}{71}{185}
\uput{5}[130](72,6){$\gam_{n-1}$}
\end{pspicture}
\end{center}
\caption{Constructing $H_n$ from $H_{n-1}$}
\label{Fig:hn}
\end{figure}

Now we will construct a sequence of hinges 
$\olH_n := H_{\ol p_n}(\ol x_n,\ol y_n)$ in $\M^2_\kap$ such that
$|\ol p_n\ol x_n| = |p_nx_n|$, $|\ol p_n\ol y_n| = |p_ny_n|$,
\begin{equation} \label{eq:xy}
|\ol x_0\ol y_0| \ge |\ol x_1\ol y_1| \ge |\ol x_2\ol y_2| \ge \dots,
\end{equation}
and such that the angle $\b\gam_n$ of $\olH_n$ is greater than or equal to 
$\gam_n$. Let $\olH_0$ be a comparison hinge for $H_0$, thus
$|\ol p_0\ol x_0| = b$, $|\ol p_0\ol y_0| = a$, and $\b\gam_0 = \gam_0$. 
For $n \ge 1$, given $\olH_{n-1}$, let $\ol p_n \in \ol p_{n-1} \ol y_{n-1}$ 
be the point at distance $b'$ from $\ol y_{n-1}$, put $\ol x_n := \ol y_{n-1}$, 
and choose $\ol y_n$ such that $(\ol p_{n-1},\ol p_n,\ol y_n)$
is a comparison triple for $(p_{n-1},p_n,y_n)$. This determines $\olH_n$.
Put $\b\om_n := \ang_{\ol p_{n-1}}(\ol p_n,\ol y_n) 
= \ang_{\ol p_{n-1}}(\ol x_n,\ol y_n)$. See Figure~\ref{Fig:bhn}.
Since $\per(p_{n-1},p_n,y_n) < \tfrac45 \per(H_0)$ and 
$y_n \in \{x_0,y_0\}$, the inequalities $\gam_{n-1} \ge \b\om_n$ and
$\pi - \gam_n \ge \pi - \b\gam_n$ hold by assumption. 
Hence, $\b\gam_{n-1} \ge \gam_{n-1} \ge \b\om_n$ and so 
$|\ol x_{n-1}\ol y_{n-1}| \ge |\ol x_n\ol y_n|$ by Lemma~\ref{Lem:c-ab}.
\begin{figure}
\begin{center}
\psset{unit=1mm,linewidth=0.8pt}
\begin{pspicture}(0,0)(90,30)
\psdots(9,4)(31,4)(70,4)(77,25)(84,21)
\psline(9,4)(70,4)(84,21)
\psline(31,4)(77,25)
\psline[linestyle=dotted,linewidth=1.2pt](70,4)(77,25)
\uput[90](10,4){$\ol x_n = \ol y_{n-1}$}
\uput[-90](20,4){$b'$}
\uput[-90](31,4){$\ol p_n$}
\uput[-45](70,4){$\ol p_{n-1}$}
\uput[45](84,21){$\ol x_{n-1}$}
\uput[45](77,25){$\ol y_n$}
\psarc[linewidth=0.3pt](31,4){4}{25}{180}
\uput{5}[105](31,4){$\b\gam_n$}
\psarc[linewidth=0.3pt](31,4){7}{0}{25}
\uput{9}[12](31,4){$\pi-\b\gam_n$}
\psarc[linewidth=0.3pt](70,4){4}{53}{180}
\uput{3}[0](70,5){$\b\gam_{n-1}$}
\psarc[linewidth=0.3pt](70,4){6}{71}{180}
\uput{7}[130](70,4){$\b\om_n$}
\end{pspicture}
\end{center}
\caption{Constructing $\olH_n$ from $\olH_{n-1}$}
\label{Fig:bhn}
\end{figure}

Now we can easily conclude the proof.
For $n \to \infty$, we have
\[
|\ol p_{n-1}\ol p_n| + |\ol p_{n-1}\ol y_n| - |\ol p_n\ol y_n| 
= l_{n-1} - l_n \to 0
\]
by~\eqref{eq:l}, consequently $\b\om_n \to \pi$ and $\b\gam_n \to \pi$ 
(note that $|\ol p_{n-1}\ol p_n| \ge a - b - 2b' > 0$ and 
$|\ol p_{n-1}\ol y_n| = b' > 0$ for $n \ge 2$).
This implies in turn that
\[
l_n - |\ol x_n\ol y_n| 
= |\ol p_n \ol x_n| + |\ol p_n \ol y_n| - |\ol x_n\ol y_n| \to 0
\]
as $n \to \infty$ (recall that $l_n \le a + b < D_\kap$).
In view of~\eqref{eq:xy} and~\eqref{eq:l}, this gives 
$|\ol x_0\ol y_0| \ge |x_0y_0|$, so $H_0$ 
satisfies (H$_\kap$) and hence also (A$_\kap$).
\end{proof}

\begin{theorem} \label{Thm:top}
Let $\kap \in \R$, and let $M$ be a complete metric space of curvature 
$\ge \kap$ in the sense of Alexandrov. Suppose that 
every pair of points in $M$ at distance $< D_\kap$ is connected by a segment.
Then every hinge $H_p(x,y)$ in $M$ with $\per(p,x,y) < 2D_\kap$
satisfies {\rm (A$_\kap$)}, {\rm (H$_\kap$)}, and {\rm (D$_\kap$)}.
\end{theorem}

\begin{proof}
Recall that by Lemma~\ref{Lem:a-d} all segments in $M$ are balanced;
furthermore, it suffices to prove that every hinge in $M$ with perimeter
less than $2D_\kap$ satisfies (A$_\kap$).  
Suppose to the contrary that there exists a hinge $H$ in $M$ 
with $\per(H) < 2D_\kap$ that does not satisfy (A$_\kap$).
Then, by Proposition~\ref{Prop:45}, there exists a hinge $H_1$ with 
$\per(H_1) < \frac45\per(H)$ and an endpoint on the union of the sides 
of $H$ such that $H_1$ does not satisfy (A$_\kap$) either. 
Inductively, for $n = 2,3,\ldots$, there exist hinges $H_n$ such that 
$\per(H_n) < \frac45\per(H_{n-1}) < \bigl(\frac45\bigr)^n\per(H)$, 
some endpoint of $H_n$ lies on the union of the sides of $H_{n-1}$, 
and $H_n$ does not satisfy~(A$_\kap$). Let $p_n$ denote the vertex of $H_n$. 
Clearly the sequence $(p_n)$ is Cauchy and thus converges 
to a point $q \in M$. However, since $M$ has curvature $\ge \kap$,
all hinges with vertex and endpoints in an appropriate neighborhood
of $q$ satisfy (A$_\kap$). This gives a contradiction, as $p_n \to q$ and
$\per(H_n) \to 0$.
\end{proof}


\section*{Appendix: Trigonometry of model spaces} 

In this appendix, we collect some trigonometric formulae for the
model spaces $\M^2_\kap$, stated in a unified way for all $\kap \in \R$
in terms of the generalized sine and cosine functions.
 
For $\kap\in\R$ we denote by $\sn_\kap\colon\R\to\R$ and 
$\cs_\kap\colon\R\to\R$ the solutions of the second
order differential equation $f'' + \kap f = 0$ satisfying the initial 
conditions 
\[
\sn_\kap(0) = 0, \quad \sn'_\kap(0) = 1, \qquad 
\cs_\kap(0) = 1, \quad \cs'_\kap(0) = 0. 
\]
Explicitly,
\begin{align*}
\sn_\kap(x)
&= \sum_{n=0}^\infty \frac{(-\kap)^n}{(2n + 1)!} x^{2n+1}
= \begin{cases}
    \sin(\sqrt{\kap}x)/\sqrt{\kap}    & \text{if $\kap > 0$,} \\
    x                                 & \text{if $\kap = 0$,} \\
    \sinh(\sqrt{-\kap}x)/\sqrt{-\kap} & \text{if $\kap < 0$,}
  \end{cases} \\
\cs_\kap(x) 
&= \sum_{n=0}^\infty \frac{(-\kap)^n}{(2n)!} x^{2n}
= \begin{cases}
    \cos(\sqrt{\kap}x)   & \text{if $\kap > 0$,} \\
    1                    & \text{if $\kap = 0$,} \\
    \cosh(\sqrt{-\kap}x) & \text{if $\kap < 0$.}
  \end{cases}
\end{align*} 
Note that 
\[
\sn'_\kap = \cs_\kap, \quad \cs'_\kap  = -\kap\sn_\kap, 
\]
and
\begin{equation}
\cs_\kap^2 + \kap\sn_\kap^2 = 1. \label{squares}
\end{equation}
The following functional equations hold. For $x,y\in\R$,
\begin{align}
\sn_\kap(x + y) &= \sn_\kap(x)\cs_\kap(y) + \cs_\kap(x)\sn_\kap(y), 
\label{snxplusy} 
\\
\cs_\kap(x + y) &= \cs_\kap(x)\cs_\kap(y) - \kap\sn_\kap(x)\sn_\kap(y); 
\label{csxplusy}
\end{align}
in particular,
\begin{align}
\sn_\kap(2x) &= 2\sn_\kap(x)\cs_\kap(x),     \label{sntwox} \\
\cs_\kap(2x) &= \cs_\kap^2(x) - \kap\sn_\kap^2(x) \label{cstwox} \\
           &= 2\cs_\kap^2(x) - 1           \nonumber \\
           &= 1 - 2\kap\sn_\kap^2(x).     \nonumber
\end{align}
Replacing $x$ by $x/2$ in the last three lines one gets
\begin{align}
\kap\sn_\kap^2 \Bigl(\frac{x}{2}\Bigr) 
&= \frac{1-\cs_\kap(x)}{2}, \label{snxhalf} \\
\cs_\kap^2 \Bigl(\frac{x}{2}\Bigr)   
&= \frac{1+\cs_\kap(x)}{2}. \label{csxhalf} 
\end{align}
Karcher~\cite{Kar} defined a ``modified distance function'' 
$\md_\kap \colon \R_+ \to \R_+$ by
\[
\md_\kap(x) 
  := \int_0^x \sn_\kap(t) \,dt 
   = \begin{cases}
       (1 - \cs_\kap (x))/\kap & \text{if $\kap \neq 0$,} \\
       x^2/2             & \text{if $\kap = 0$.}
     \end{cases}
\]
In view of~(\ref{snxhalf}), this can be written as
\[
\md_\kap(x) = 2 \sn_\kap^2\Bigl(\frac{x}{2}\Bigr).
\]
It is easy to check that
\begin{align}
\cs_\kap + \kap \md_\kap 
&= 1, \label{csplusmd} \\
\md_\kap(x + y) 
&= \md_\kap(x - y) + 2 \sn_\kap(x) \sn_\kap(y) \label{mdxplusy} \\
&= \md_\kap(x) + \cs_\kap(x) \md_\kap(y) + \sn_\kap(x) \sn_\kap(y) \nonumber \\
&= \md_\kap(x) \cs_\kap(y) + \md_\kap(y) + \sn_\kap(x) \sn_\kap(y), \nonumber \\
\md_\kap(2x) 
&= 2 \sn_\kap^2(x) \label{mdtwox} \\
&= 2(1 + \cs_\kap(x)) \md_\kap(x). \nonumber
\end{align}  

We turn to trigonometry. 
Consider a triangle in $\M^2_\kap$ with vertices $x,y,z$ and 
(possibly degenerate) sides of length $a,b,c \ge 0$, 
where $a = |yz|$, $b = |zx|$, and $c = |xy|$, and let
$\alpha,\beta,\gamma \in [0,\pi]$ denote the angles at $x,y,z$, respectively,
whenever they are defined. 
The law of cosines can be stated in a unified way~as 
\begin{align}
\md_\kap(c) 
&= \md_\kap(a+b) - \sn_\kap(a) \sn_\kap(b) (1 + \cos(\gamma)) \label{eq:cos1} \\
&= \md_\kap(a-b) + \sn_\kap(a) \sn_\kap(b) (1 - \cos(\gamma)) \nonumber \\
&= \md_\kap(a) + \cs_\kap(a) \md_\kap(b) 
    - \sn_\kap(a) \sn_\kap(b) \cos(\gamma) \nonumber \\        
&= \md_\kap(a) \cs_\kap(b) + \md_\kap(b) 
    - \sn_\kap(a) \sn_\kap(b) \cos(\gamma) \nonumber
\end{align}
(compare~(\ref{mdxplusy})), 
or, in terms of $\sn_\kap$, as
\begin{align}
\sn_\kap^2 \Bigl(\frac{c}{2}\Bigr) 
&= \sn_\kap^2 \Bigl(\frac{a+b}{2}\Bigr) - \sn_\kap(a)\sn_\kap(b) 
\cos^2 \Bigl(\frac{\gamma}{2}\Bigr) \label{eq:snc} \\
&= \sn_\kap^2 \Bigl(\frac{a-b}{2}\Bigr) + \sn_\kap(a)\sn_\kap(b) 
\sin^2 \Bigl(\frac{\gamma}{2}\Bigr). \nonumber
\end{align}
Multiplying any of these equations by $\kap$ one obtains the more familiar
formula 
\begin{equation}
\cs_\kap(c) = \cs_\kap(a)\cs_\kap(b) + \kap\sn_\kap(a)\sn_\kap(b) \cos(\gamma) 
\end{equation}
for the hyperbolic and spherical geometries. 
The ``dual law of cosines'' or ``law of cosines for angles'' is the identity
\begin{equation}
\cos(\gam) = \sin(\alpha)\sin(\beta)\cs_\kap(c) - \cos(\alpha)\cos(\beta); 
\end{equation}
in the Euclidean case it represents the fact
that $\alpha + \beta + \gamma = \pi$. 
The law of sines is given by
\begin{equation}
\sn_\kap(a) \sin(\beta) = \sn_\kap(b) \sin(\alpha).
\end{equation}
Let $l$ denote the distance from the midpoint of the side $xy$
of the triangle to the vertex $z$. Then
\begin{equation}
2 \cs_\kap \Bigl(\frac{c}{2}\Bigr) \md_\kap l 
= \md_\kap(a) + \md_\kap(b) - 2\md_\kap\Bigl(\frac{c}{2}\Bigr);
\label{mdl}
\end{equation}
equivalently,
\begin{equation}
2\cs_\kap \Bigl(\frac{c}{2}\Bigr) \sn_\kap^2 \Bigl(\frac{l}{2}\Bigr) 
= \sn_\kap^2 \Bigl(\frac{a}{2}\Bigr) +
\sn_\kap^2 \Bigl(\frac{b}{2}\Bigr) - 2\sn_\kap^2\Bigl(\frac{c}{4}\Bigr).
\end{equation}
(This equation may be used to define spaces of curvature $\ge \kap$ or
$\le \kap$.) Multiplying by $\kap$ one obtains the simple formula
\begin{equation}
2\cs_\kap \Bigl(\frac{c}{2}\Bigr) \cs_\kap(l) = \cs_\kap(a) + \cs_\kap(b)
\end{equation}
for the hyperbolic and spherical geometries.

\begin{proof}[Proof of~\eqref{mdl}]
(We omit all subscripts $\kap$.) By~\eqref{eq:cos1},
\begin{align*}
\md(l) &= \md(b) \cs\Bigl(\frac{c}{2}\Bigr) + \md\Bigl(\frac{c}{2}\Bigr) 
- \sn(b) \sn\Bigl(\frac{c}{2}\Bigr) \cos(\alpha), \\
\md(a) &= \md(b) \cs(c) + \md(c) - \sn(b) \sn(c) \cos(\alpha).
\end{align*}
Using~(\ref{csxhalf}) and~(\ref{sntwox}) we get
\[
2 \cs\Bigl(\frac{c}{2}\Bigr) \md(l) - \md(a) 
= \md(b) + 2 \cs\Bigl(\frac{c}{2}\Bigr) \md\Bigl(\frac{c}{2}\Bigr) - \md(c).
\]
Now the formula follows from~(\ref{mdtwox}).         
\end{proof}



\end{document}